\documentclass[11pt,leqno]{article}
\usepackage{amsmath,amssymb,amsthm,latexsym}
\usepackage{indentfirst}
\usepackage{amsmath}

\newtheorem{theorem}{\bf \large Theorem}[section]

\newtheorem{lemma}{\bf \large Lemma}[section]
\newtheorem{ex}{\bf \large Example}[section]
\newtheorem{remark}{\bf \large Remark}[section]

\title{\textbf{A M\"{o}bius scalar curvature rigidity on compact  conformally flat  hypersurfaces in $\mathbb{S}^{n+1}$}}
\author {\normalsize {Limiao Lin$^a$, Tongzhu Li$^b$, Changping Wang$^a$\thanks{The work  is  supported by the grant No.
11571037 and No.11471021 of NSFC.}}, \\
\small{$^a$ College of Mathematics and Informatics, FJKLMAA, Fujian Normal University,}\\
\small{Fuzhou, 350108, China,}\\
\small{$^b$ Department of Mathematics, Beijing Institute of
Technology,} \\
\small{Beijing,100081,China,},\\
\small{E-mail: 83343055@163.com,~~ litz@bit.edu.cn,~~ cpwang@fjnu.edu.cn.}}

\date{}

\begin{document}
\maketitle
\begin{abstract}
In this paper, we study  conformally flat hypersurfaces of dimension $n(\geq 4)$ in  $\mathbb{S}^{n+1}$
using the framework of M\"{o}bius geometry.
First, we classify and explicitly express   the conformally flat  hypersurfaces of dimension $n(\geq 4)$ with constant M\"{o}bius scalar curvature
 under the M\"{o}bius transformation group of $\mathbb{S}^{n+1}$. Second, we prove
that if the  conformally flat  hypersurface with constant M\"{o}bius scalar curvature $R$ is compact,
then $$R=(n-1)(n-2)r^2, ~~0<r<1,$$ and the compact conformally flat hypersurface is M\"{o}bius equivalent to the torus
$$\mathbb{ S}^1(\sqrt{1-r^2})\times \mathbb{S}^{n-1}(r)\hookrightarrow \mathbb{S}^{n+1}.$$
\end{abstract}

\vskip 1 cm
\section{Introduction}
A Riemannian manifold $(M^n, g)$ is conformally
flat, if every point has a neighborhood which is conformal to an open
set in the Euclidean space $\mathbb{R}^n$. A hypersurface of the sphere
$\mathbb{S}^{n+1}$ is said to be conformally flat if so it is with respect to the induced
metric. Due to conformal invariant objects,  the theory of conformally flat hypersurfaces
is essentially the same whether it is considered in the space forms $\mathbb{R}^{n+1}$, $\mathbb{S}^{n+1}$ or $\mathbb{H}^{n+1}$.
In fact,
there exists conformal
diffeomorphism between the space forms.
The $(n+1)$-dimensional hyperbolic space $\mathbb{H}^{n+1}$ defined by
$$\mathbb{H}^{n+1}=\{(y_0,y_1,\cdots,y_{n+1})|-y_0^2+y_1^2+\cdots+y_{n+1}^2=-1,y_0>0\}.$$
The conformal
diffeomorphisms $\sigma,\tau$ are defined by
\begin{equation*}
\begin{split}
&\sigma:\mathbb{R}^{n+1}\rightarrow \mathbb{S}^{n+1}\backslash
\{(-1,\vec{0})\},~~~~\sigma(u)=(\frac{1-|u|^2}{1+|u|^2},\frac{2u}{1+|u|^2}),\\
&\tau:\mathbb{H}^{n+1}\rightarrow \mathbb{S}^{n+1}_+\subset \mathbb{S}^{n+1},
~~~\tau(y)=(\frac{1}{y_0},\frac{\vec{y}}{y_0}), ~~y=(y_0,\vec{y})\in
\mathbb{H}^{n+1},
\end{split}
\end{equation*}
where $\mathbb{S}^{n+1}_+$ is the hemisphere in $\mathbb{S}^{n+1}$ whose the first
coordinate is positive.
By conformal
diffeomorphisms $\sigma,\tau$,
the conformally flat hypersurfaces
in space forms are equivalent to each other.

The dimension of the hypersurface seems to play an important role
in the study of conformally flat hypersurfaces. For $n\geq4$, the immersed hypersurface  $f: M^n \to \mathbb{S}^{n+1}$ is conformally flat
if and only if at least $n-1$ of the principal curvatures coincide
at each point by the result of Cartan-Schouten
(\cite{c},\cite{schou}). Cartan-Schouten's result is no longer true in dimension
$3$. Lancaster (\cite{lan}) gave some examples of conformally flat
hypersurfaces in $\mathbb{R}^4$ having three different principal curvatures. For $n=2$,
the existence of isothermal coordinates means that any Riemannian
surface is conformally flat.
Do Carmo, Dajczer and Mercuri in \cite{do} have studied Diffeomorphism
types of the compact conformally flat hypersurfaces in $\mathbb{R}^{n+1}$.
Pinkall in \cite{pin1} was studied the intrinsic conformal geometry of compact conformally flat hypersurfaces.
Suyama in \cite{su}
explicitly constructs compact conformally flat hypersurfaces in  space forms using codimension
one foliation by $(n-1)$-spheres.
 Standard examples of the conformally flat hypersurfaces come from cones, cylinders,
or rotational hypersurfaces over a curve in  Euclidean
$2$-space $\mathbb{R}^2$, $2$-sphere $\mathbb{S}^2$, or hyperbolic $2$-space $\mathbb{R}^2_+$, respectively (see section 3). In \cite{lin}, Lin and Guo showed that if the conformally flat hypersurface has closed M\"{o}bius form, then it is M\"{o}bius  equivalent to one of the standard examples.

It is known that the  conformal transformations group of a sphere is isomorphic to
its M\"{o}bius transformation group.  As conformal invariant objects, conformally flat hypersurfaces are investigated
in this paper using the framework of M\"{o}bius geometry. If the conformally flat hypersurface is no umbilical point everywhere, then there
exists a global M\"{o}bius metric (see section 2), which is invariant under the M\"{o}bius
transformation group of $\mathbb{S}^{n+1}$.  The scalar curvature with respect to the M\"{o}bius
metric is called M\"{o}bius scalar curvature.
First, we classify  locally   the conformally flat  hypersurfaces of dimension $n(\geq 4)$ with constant M\"{o}bius scalar curvature
 under the M\"{o}bius transformation group of $\mathbb{S}^{n+1}$.

\begin{theorem}\label{the1}
Let $f:M^n\rightarrow \mathbb{S}^{n+1}$, $n\geq 4$, be a conformally flat hypersurface
without umbilical points. If the M\"{o}bius scalar curvature is constant,
then the M\"{o}bius form is closed and $f$  is M\"{o}bius equivalent to one of the following hypersurfaces in $\mathbb{S}^{n+1}$,\\
$(i)$ the image of $\sigma$ of a cylinder over a curvature-spiral in  $\mathbb{R}^2\subset \mathbb{R}^{n+1}$;\\
$(ii)$ the image of $\sigma$ of a cone over a curvature-spiral in  $\mathbb{S}^2\subset \mathbb{R}^3\subset \mathbb{R}^{n+1}$; \\
$(iii)$ the image of $\sigma$ of a rotational hypersurface over a curvature-spiral in  $\mathbb{R}^2_+\subset \mathbb{R}^{n+1}$.\\
\end{theorem}
Here the so-called \emph{curvature-spiral} in a 2-dimensional space form
$N^2(\epsilon)=\mathbb{S}^2,\mathbb{R}^2,\mathbb{R}^2_+$ (of Gaussian curvature $\epsilon=1,0,-1$ respectively)
is determined by the intrinsic equation
\begin{equation}\label{spiral}
 -\frac{\kappa_{ss}}{\kappa^3}+\frac{(n+2)\kappa_s^2}{2\kappa^4}+\epsilon\frac{n-2}{2\kappa^2}
=R,~~~~\kappa_s=\frac{d}{ds}\kappa.
\end{equation}
Here $s$ is the arc-length parameter, $\kappa$ denotes the
geodesic curvature of the curve $\gamma$, and $R$ is a real constant. In \cite{guo}, authors classified locally the hypersurfaces with constant
M\"{o}bius sectional curvature, which is some special conformally flat hypersurfaces with M\"{o}bius scalar curvature by the equation (\ref{spiral}).

For compact conformally flat hypersurfaces, we obtain the following M\"{o}bius scalar curvature rigidity theorem, which means that the closed curve in $\mathbb{R}^2_+$ satisfying the intrinsic equation (\ref{spiral}) with geodesic curvature $\kappa>0$ is circle  $\mathbb{S}^1$.
\begin{theorem}\label{the2}
Let $f:M^n\rightarrow \mathbb{S}^{n+1}$, $n\geq 4$, be a compact conformally flat hypersurface
without umbilical points everywhere. If the M\"{o}bius scalar curvature $R$ is constant,
then $$R=(n-1)(n-2)r^2, ~~0<r<1,$$ and the compact conformally flat hypersurface is M\"{o}bius equivalent to the torus
$$f:\mathbb{ S}^1(\sqrt{1-r^2})\times \mathbb{S}^{n-1}(r)\to \mathbb{S}^{n+1}.$$
\end{theorem}
\begin{remark}
Theorem \ref{the1} and Theorem \ref{the2} is true for $n=3$ provided that the $3$-dimensional conformally flat hypersurface has only two distinct principal
curvatures.
\end{remark}

The paper is organized as follows. In section 2, we review the
elementary facts about M\"{o}bius geometry of hypersurfaces in
$\mathbb{S}^{n+1}$. In section 3, we prove the theorem \ref{the1}. In section 4, we prove the theorem \ref{the2}.

\vskip 1 cm
\section{M\"{o}bius invariants of hypersurfaces in $\mathbb{S}^{n+1}$}
In this section, we recall some facts about the M\"{o}bius invariants of hypersurfaces in $\mathbb{S}^{n+1}$.
For details we refer to \cite{w}.

Let $f:M^{n}\rightarrow \mathbb{S}^{n+1}$ be a hypersurface without umbilical points. In this section we use the range of indices: $1\leq i,j,k,l\leq n.$ We
assume that $\{e_i\}$ is an orthonormal basis with respect to the
induced metric with $\{\theta_i\}$ the dual basis.
Let $II=\sum_{ij}h_{ij}\theta_i\theta_j$ and
$H=\sum_i\frac{h_{ii}}{n}$ be the second fundamental form and the
mean curvature of $x$, respectively. We define the M\"{o}bius metric $g$, the M\"{o}bius second fundamental form $B$, the Blaschke tensor $A$ and the M\"{o}bius form $C$
as follows, respectively,
\begin{equation}\label{moebius}
\begin{split}
g=&\rho^2dx\cdot dx,~~~~~~~\rho^2=\frac{n}{n-1}(|h|^2-nH^2),\\
B=&\rho\sum_{ij}(h_{ij}-H\delta_{ij})\theta_i\otimes\theta_j,\\
C=&-\rho^{-1}\sum_i[e_i(H)+\sum_j(h_{ij}-H\delta_{ij})e_j]\theta_i,\\
A=&\sum_{ij}\Big\{e_i(\log\rho)e_j(\log\rho)-\nabla_{e_i}\nabla_{e_j}\log\rho+Hh_{ij}+\\
&\frac{1}{2}[1-H^2-|\bigtriangledown\log\rho|^2]\delta_{ij}\Big\}\theta_i\otimes\theta_j.
\end{split}
\end{equation}
Note that the conformal compactification space
$\mathbb{S}^{n+1}$ unifies the space forms $\mathbb{S}^{n+1},$ $\mathbb{R}^{n+1},\mathbb{H}^{n+1}$ and the
formula above defining the M\"obius metric $g$ and the M\"{o}bius second fundamental form $B$ are the same for any of them.
\begin{theorem}$\cite{w}$\label{fundthe}
Two hypersurfaces $f: M^n\rightarrow \mathbb{S}^{n+1}$ and $\bar{f}:
M^n\rightarrow \mathbb{S}^{n+1} (n\geq 3)$ are M\"{o}bius equivalent if and
only if there exists a diffeomorphism $\varphi: M^n\rightarrow M^n$
which preserves the M\"{o}bius metric and the M\"{o}bius second
fundamental form.
\end{theorem}

Let $E_i=\rho^{-1}e_i, \omega_i=\rho\theta_i$, then $\{E_1,\cdots,E_n\}$ is an orthonormal basis with respect to the
M\"{o}bius metric $g$ with the dual basis $\{\omega_1,\cdots,\omega_n\}$. Let $\{\omega_{ij}\}$ be the connection $1$-form of the
M\"{o}bius metric under the orthonormal basis $\{\omega_i\}$, and

$$
A=\sum_{ij}A_{ij}\omega_i\otimes\omega_j,~~
B=\sum_{ij}B_{ij}\omega_i\otimes\omega_j,~~
C=\sum_iC_i\omega_i.$$
The covariant derivative of
$C_i, A_{ij}, B_{ij}$ are defined by
\begin{eqnarray*}
&&\sum_jC_{i,j}\omega_j=dC_i+\sum_jC_j\omega_{ji},\\
&&\sum_kA_{ij,k}\omega_k=dA_{ij}+\sum_kA_{ik}\omega_{kj}+\sum_kA_{kj}\omega_{ki},\\
&&\sum_kB_{ij,k}\omega_k=dB_{ij}+\sum_kB_{ik}\omega_{kj}+\sum_kB_{kj}\omega_{ki}.
\end{eqnarray*}
The integrability conditions of the M\"{o}bius invariants are given
by
\begin{eqnarray}
&&A_{ij,k}-A_{ik,j}=B_{ik}C_j-B_{ij}C_k,\label{equa1}\\
&&C_{i,j}-C_{j,i}=\sum_k(B_{ik}A_{kj}-B_{jk}A_{ki}),\label{equa2}\\
&&B_{ij,k}-B_{ik,j}=\delta_{ij}C_k-\delta_{ik}C_j,\label{equa3}\\
&&R_{ijkl}=B_{ik}B_{jl}-B_{il}B_{jk}
+\delta_{ik}A_{jl}+\delta_{jl}A_{ik}
-\delta_{il}A_{jk}-\delta_{jk}A_{il},\label{equa4}.
\end{eqnarray}
where $R_{ijkl}$ denote the curvature tensor of $g$.
Moreover,
\begin{equation}\label{ba}
\begin{split}
&\sum_iB_{ii}=0, ~~\sum_{ij}(B_{ij})^2=\frac{n-1}{n},\\
&\sum_iA_{ii}=\frac{1}{2n}+\frac{R}{2(n-1)},~~\sum_jB_{ij,j}=-(n-1)C_i,
\end{split}
\end{equation}
where $R=\sum_{i>j}R_{ijij}$ is the M\"{o}bius scalar curvature.

By equation (\ref{equa2}), we have
\begin{equation}\label{cb}
dC=0\Leftrightarrow\sum_k(B_{ik}A_{kj}-B_{jk}A_{ki})=0,
\end{equation}
which implies that the matrix $(B_{ij})$ and $(A_{ij})$ can be diagonalizable simultaneously.

\vskip 1 cm
\section{Local geometry of conformally flat hypersurfaces}

In this section,  we will give the M\"{o}bius invariants of the standard examples of  conformally flat hypersurfaces in
$\mathbb{R}^{n+1}$.  Then we prove that the conformally flat hypersurfaces with constant M\"{o}bius scalar curvature
come from these examples.

A key observation is that the M\"obius metric of those standard examples are of the form
\[g=\kappa^2(s)\left(ds^2+I_{-\epsilon}^{n-1}\right)~,\]
where $I_{-\epsilon}^{n-1}$ is the metric of $n-1$ dimensional space form
of constant curvature $-\epsilon$. For such metric forms we have
\begin{lemma}\label{metric}
The metric $g=\kappa^2(s)(ds^2+I_{-\epsilon}^{n-1})$ given above is
of constant scalar curvature $R$ if and only if the function $\kappa(s)$ satisfies
$$ -\frac{\kappa_{ss}}{\kappa^3}+\frac{(n+2)\kappa_s^2}{2\kappa^4}+\epsilon\frac{n-2}{2\kappa^2}
=R,~~~~\kappa_s=\frac{d}{ds}\kappa.$$
\end{lemma}
This lemma is easy to prove using exterior differential forms and we omit the proof
at here. Below we give the explicit construction of the standard examples of conformally flat hypersurfaces
as well as their M\"{o}bius metric.
\begin{ex}\label{ex1}
Let $\gamma:I\rightarrow \mathbb{R}^2$ be a regular curve, and
$s$ denote the arclength of $\gamma(s)$. we define cylinder  in
$\mathbb{R}^{n+1}$ over $\gamma$,
$$f(s,y)=(\gamma(s),y):I\times \mathbb{R}^{n-1}\longrightarrow \mathbb{R}^{n+1},$$
where $y:\mathbb{R}^{n-1}\longrightarrow \mathbb{R}^{n-1}$ is identical maping.
\end{ex}
The first fundamental form $I$ and the second
fundamental form $II$ of the cylinder $f$ are, respectively,
\[ I=ds^2+I_{\mathbb{R}^{n-1}}, \;\; II=\kappa ds^2~,\]
where $\kappa(s)$ is the geodesic curvature of $\gamma$,  $I_{\mathbb{R}^{n-1}}$ is the standard Euclidean metric of $\mathbb{R}^{n-1}$.
So we have
$(h_{ij})=\operatorname{diag}(\kappa,0,\cdots,0)~,
~H=\frac{\kappa}{n}~,~\rho=\kappa~.$
Thus the M\"{o}bius metric $g$ of the cylinder $f$ is
\[g=\rho^2 I=\kappa^2(ds^2+I_{\mathbb{R}^{n-1}}).\]
where $I_{\mathbb{R}^{n-1}}$ is the standard hyperbolic metric of $\mathbb{R}^{n-1}(-1)$.
Because  at least $n-1$ of the principal curvatures coincide
at each point, the cylinder $f$ is a conformally flat hypersurface.
When $\gamma=\mathbb{S}^1$, the cylinder $f$ is the isoparametric
hypersurface $\mathbb{S}^1\times \mathbb{R}^{n-1}\to \mathbb{R}^{n+1}$.

\begin{ex}\label{ex2}
Let $\gamma:I\rightarrow \mathbb{S}^2(1)\subset \mathbb{R}^3$ be a regular
curve, and $s$ denote the arclength of $\gamma(s)$. we define
cone in $R^{n+1}$ over $\gamma$,
$$f(s,t,y)=(t\gamma(s),y):I\times \mathbb{R}^{+}\times \mathbb{R}^{n-2}\longrightarrow \mathbb{R}^{n+1},$$
where $y:\mathbb{R}^{n-2}\longrightarrow \mathbb{R}^{n-2}$ is identical mapping and
$\mathbb{R}^{+}=\{t|t>0\}$.
\end{ex}
The first and second
fundamental forms of the cone $f$ are, respectively,
\[I=t^2ds^2+I_{R^{n-1}}~, \;\; II=t\kappa ds^2~.\]
So we have
$(h_{ij})=\operatorname{diag}
\left(\frac{\kappa}{t},0,\cdots,0\right)~,~H=\frac{\kappa}{nt}~,
~\rho=\frac{\kappa}{t}~.$
Thus the M\"{o}bius metric $g$ of the cone $f$ is
\[g=\rho^2I=\frac{\kappa^2}{t^2}\left(t^2ds^2+I_{\mathbb{R}^{n-1}}\right)
=\kappa^2(ds^2+I_{\mathbb{H}^{n-1}})~,\]
where $I_{\mathbb{H}^{n-1}}$ is the standard hyperbolic metric of $\mathbb{H}^{n-1}(-1)$.
Clearly the cone $f$ is a conformally flat hypersurface.  When $\gamma=\mathbb{S}^1$, the cone $f$ is the image of $\tau^{-1}\circ\sigma$ of the isoparametric
hypersurface $\mathbb{S}^1(r)\times \mathbb{H}^{n-1}(\sqrt{1+r^2})\to \mathbb{H}^{n+1}$.

\begin{ex}\label{ex3}
Let $\mathbb{R}^2_+=\{(x,y)\in \mathbb{R}^2|y>0\}$ be the upper half-space endowed with the standard
hyperbolic metric
\[
ds^2=\frac{1}{y^2}[dx^2+dy^2]~.
\]
Let  $\gamma=(x,y):I\longrightarrow \mathbb{R}^2_+$
be a regular curve, and $s$ denote the arclength of $\gamma(s)$. we define rotational hypersurface in $\mathbb{R}^{n+1}$ over $\gamma$,
\begin{equation*}
f:I\times \mathbb{S}^{n-1}\longrightarrow \mathbb{R}^{n+1},~~~~~
f(x,y,\theta)=(x,y\theta),
\end{equation*}
where $\theta:\mathbb{S}^{n-1}\longrightarrow \mathbb{R}^{n}$ is a standard immersion of a round sphere.
\end{ex}
In the Poincare half plane $\mathbb{R}^2_+$ we denote the covariant differential
of the hyperbolic metric as $D$. Choose orthonormal frames
$e_1=y\frac{\partial}{\partial x},e_2=y\frac{\partial}{\partial y}$. It is easy to find
\[D_{e_1}e_1=e_2~,~D_{e_1}e_2=-e_1~,~D_{e_2}e_1=D_{e_2}e_2=0.\]
For $\gamma(s)=((x(s),y(s))\subset \mathbb{R}^2_+$ let $x'$ denote derivative $\partial x/\partial s$ and so on. Choose the unit tangent vector $\alpha=\frac{1}{y}(x'(s)e_1+y'(s)e_2)$ and the unit normal vector $\beta=\frac{1}{y}(-y'(s)e_1+x'(s)e_2)$. The geodesic curvature is computed via
\[\kappa=\langle D_\alpha \alpha,\beta\rangle
=\frac{x'y''-x''y'}{y^2}+\frac{x'}{y}~.\]
After these preparation, we see that the rotational hypersurface
$f(x,y,\theta)=(x,y\theta)$ has differential
$df=(x'ds,y'\theta ds+y d\theta)$ and unit normal vector $\eta=\frac{1}{y}(-y',x'\theta).$
Thus the first and second
fundamental forms of hypersurface $f$ are, respectively,
\[I=df\cdot df=y^2(ds^2+I_{\mathbb{S}^{n-1}})~,
~ II=-df\cdot d\eta=(y\kappa-x')ds^2-x'I_{\mathbb{S}^{n-1}}~,\]
where $I_{\mathbb{S}^{n-1}}$ is the standard metric of $\mathbb{S}^{n-1}(1)$.
Thus principal curvatures are \newline
$\frac{\kappa y-x'}{y^2},\frac{-x'}{y^2},
\cdots,\frac{-x'}{y^2}.$
So $\rho=\frac{\kappa}{y}$, and the M\"{o}bius metric of $f$ is
\[g=\rho^2I=\kappa^2(ds^2+I_{\mathbb{S}^{n-1}}).\]

Clearly the hypersurface $f$ is  a conformally flat hypersurface. When $\gamma=\mathbb{S}^1$, the cone $f$ is the image of $\sigma$ of the isoparametric
hypersurface $\mathbb{S}^1(\sqrt{1-r^2})\times \mathbb{S}^{n-1}(r)\to \mathbb{S}^{n+1}$.

Next,  we compute the M\"{o}bius invariant of the conformally flat hypersurfaces. From (\ref{ba}),
We can choose a local orthonormal basis
$\{E_1,\cdots,E_n\}$ with respect to the M\"{o}bius metric $g$ such
that
$$(B_{ij})=diag(\frac{n-1}{n},\frac{-1}{n},\cdots,\frac{-1}{n}).$$
In the following section we make use of the following convention on the
ranges of indices:
$$1\leq i,j,k \leq n; ~~2\leq \alpha,\beta,\gamma \leq n.$$
Since $B_{\alpha\beta}=\frac{1}{n}\delta_{\alpha\beta}$, we can
rechoose a local orthonormal basis $\{E_1,\cdots,E_n\}$ with
respect to the M\"{o}bius metric $g$ such that
\begin{equation}\label{bba}
(B_{ij})=diag(\frac{n-1}{n},\frac{-1}{n},\cdots,\frac{-1}{n}),~~
(A_{ij})=\left(
  \begin{array}{ccccc}
    A_{11} & A_{12} & A_{13} & \cdots & A_{1n} \\
    A_{21} & a_2 & 0 & \cdots & 0 \\
    A_{31}& 0 & a_3& \cdots & 0 \\
    \vdots & \vdots & \vdots& \ddots & \vdots \\
   A_{n1} & 0 & 0 & \cdots &a_n \\
  \end{array}
\right)
\end{equation}
Let $\{\omega_1,\cdots,\omega_n\}$ be the dual basis, and $\{\omega_{ij}\}$ the connection forms.
\begin{lemma}\label{le1}
Let $f: M^n\rightarrow \mathbb{S}^{n+1}$ $(n\geq 4)$ be a conformally flat
hypersurface without umbilical points. If the M\"{o}bius scalar curvature is constant,  then we can choose a local orthonormal
basis $\{E_1,\cdots,E_n\}$ with
respect to the M\"{o}bius metric $g$ such that
\begin{equation}\label{bbcc}
(B_{ij})=diag\{\frac{n-1}{n},\frac{-1}{n},\cdots,\frac{-1}{n}\},~~
(A_{ij})=diag\{a_1,a_2,\cdots,a_2\}.
\end{equation}
Moreover, the distribution $\mathbb{D}=span\{E_2,\cdots,E_n\}$ is integrable.
\end{lemma}
\begin{proof}
 Using
$dB_{ij}+\sum_kB_{kj}\omega_{ki}+\sum_kB_{ik}\omega_{kj}=\sum_kB_{ij,k}\omega_k$
, the equation (\ref{equa3}), we get
\begin{equation}\label{a1}
\begin{split}
&B_{1\alpha,\alpha}=-C_1, ~~otherwise, ~~~B_{ij,k}=0;\\
&\omega_{1\alpha}=-C_1\omega_{\alpha},~~~~C_{\alpha}=0.
\end{split}
\end{equation}

Thus $d\omega_1=0$ and the distribution $\mathbb{D}=span\{E_2,\cdots,E_n\}$ is
integrable.

Using $dC_i+\sum_kC_k\omega_{ki}=\sum_kC_{i,k}\omega_k$ and
(\ref{a1}), we can obtain
\begin{equation}\label{c2}
C_{\alpha,\alpha}=-C_1^2, ~~~C_{\alpha,k}=0,\alpha\neq k.
\end{equation}
From (\ref{a1}),
\begin{equation*}
\begin{split}
&d\omega_{1\alpha}=-dC_1\wedge\omega_{\alpha}-C_1d\omega_{\alpha}\\
&=-dC_1\wedge\omega_{\alpha}+C_1^2\omega_1\wedge\omega_{\alpha}-C_1\sum_{\gamma}\omega_{\gamma}\wedge\omega_{\gamma\alpha},
\end{split}
\end{equation*}
and
$d\omega_{1\alpha}-\sum_j\omega_{1j}\wedge\omega_{j\alpha}=-\frac{1}{2}\sum_{kl}R_{1\alpha
kl}\omega_k\wedge\omega_l$, we get that
\begin{equation}\label{c1}
R_{1\alpha
1\alpha}=C_{1,1}-C_1^2,~~~~R_{1\alpha\beta\alpha}-C_{1,\beta}=0.
\end{equation}
Since $R_{1\alpha
1\alpha}=-\frac{n-1}{n^2}+a_1+a_{\alpha}=C_{1,1}-C_1^2$ and
$R_{1\alpha\beta\alpha}=A_{1\beta},\alpha\neq\beta$, thus we have
\begin{equation}\label{a2}
a_2=a_3=\cdots=a_n,~~~~ A_{1\beta}=C_{1,\beta}.
\end{equation}
Thus $A|_{\mathbb{D}}=aI,~~a=a_2$. Since $E_1$ is principal vector field, then vector $E=A_{12}E_2+\cdots+A_{1n}E_n$ is well defined.
If $E=0$, then $A_{12}=\cdots=A_{1n}=0$. If $E\neq 0$,
we can rechoose a local orthonormal basis  $\{\tilde{E}_2=\frac{E}{|E|},\tilde{E}_3,\cdots,\tilde{E}_n\}$ of $\mathbb{D}$ with
respect to the M\"{o}bius metric $g$ such that
\begin{equation}\label{bba}
(B_{ij})=diag(\frac{n-1}{n},\frac{-1}{n},\cdots,\frac{-1}{n}),~~
(A_{ij})=\left(
  \begin{array}{ccccc}
    A_{11} & A_{12} & 0 & \cdots & 0 \\
    A_{21} & a_2 & 0 & \cdots & 0 \\
    0& 0 & a_2& \cdots & 0 \\
    \vdots & \vdots & \vdots& \ddots & \vdots \\
   0 & 0 & 0 & \cdots &a_2 \\
  \end{array}
\right)
\end{equation}

To finish the proof of the Lemma, we need to prove that $A_{12}=0$.
Using $dA_{ij}+\sum_kA_{kj}\omega_{ki}+\sum_kA_{ik}\omega_{kj}=\sum_kA_{ij,k}\omega_k$
, the equation (\ref{equa1}) and (\ref{bba}), we get
\begin{equation}\label{aa}
\begin{split}
&\sum_mA_{12,m}\omega_m=dA_{12}+(A_{11}-A_{22})\omega_{12},\\
&\sum_mA_{1\alpha,m}\omega_m=(A_{11}-a_2)\omega_{1\alpha}+A_{12}\omega_{2\alpha},~~\sum_mA_{2\alpha,m}\omega_m=A_{12}\omega_{1\alpha},~~\alpha\geq3,\\
&\sum_kA_{11,k}\omega_k=dA_{11}+2A_{12}\omega_{21}, ~~\sum_kA_{22,k}\omega_k=dA_{22}+2A_{12}\omega_{12},\\
&\sum_kA_{\alpha\alpha,k}\omega_k=dA_{\alpha\alpha},~~A_{\alpha\beta,k}=0, ~~\alpha\neq\beta, ~~\alpha,\beta\geq 3.
\end{split}
\end{equation}
From the fourth and seventh equation in (\ref{aa}), we get
\begin{equation}\label{aad}
E_{\alpha}(a_2)=A_{\beta\beta,\alpha}=A_{\beta\alpha,\beta}=0, ~~~\alpha\geq 3.
\end{equation}
Since the M\"{o}bius scalar curvature is constant, $tr(A)=A_{11}+(n-1)a_2$ is constant. Thus
\begin{equation}\label{aad1}
A_{1\alpha,1}=A_{11,\alpha}=E_{\alpha}(A_{11})=0,~~~\alpha\geq 3.
\end{equation}
From the first, second and third  equation in (\ref{aa}), we get
\begin{equation}\label{aad2}
A_{12,2}=E_2(A_{12})-(A_{11}-a_2)C_1,~~~A_{1\beta,\beta}=-(A_{11}-a_2)C_1+A_{12}\omega_{2\beta}(E_{\beta}).
\end{equation}
On the other hand, From (\ref{equa1}), we have
\begin{equation*}
\begin{split}
&E_1(A_{22})=A_{22,1}=A_{12,2}+\frac{1}{n}C_1=E_2(A_{12})-(A_{11}-a_2)C_1+\frac{1}{n}C_1,\\
&E_1(A_{\alpha\alpha})=A_{\alpha\alpha,1}=A_{1\alpha,\alpha}+\frac{1}{n}C_1=-(A_{11}-a_2)C_1+A_{12}\omega_{2\beta}(E_{\beta})+\frac{1}{n}C_1,\\
\end{split}
\end{equation*}
which implies that
\begin{equation}\label{2w}
E_2(A_{12})=A_{12}\omega_{2\beta}(E_{\beta}).
\end{equation}
Since $A_{1\alpha,\beta}=A_{\alpha\beta,1}=A_{1\beta,1}=A_{12,\alpha}=0,~\alpha\neq\beta$, from the second equation in (\ref{aa}) we can obtain
\begin{equation}\label{aa4}
A_{12}\omega_{2\beta}(E_k)=0, ~~\beta\geq 3,~~k\neq\beta.
\end{equation}
From the first and third  equation in (\ref{aa}), we get
\begin{equation}\label{aa1}
\begin{split}
E_2(a_2)&=E_2(A_{\beta\beta})=A_{\beta\beta,2}=A_{2\beta,\beta}=-A_{12}C_1,\\
E_1(A_{12})&=A_{12,1}=A_{11,2}=E_2(A_{11})+2A_{12}C_1\\
&=E_2(-(n-1)a_2)+2A_{12}C_1=(n+1)A_{12}C_1.
\end{split}
\end{equation}
Now we assume that $A_{12}\neq 0,$ From (\ref{2w}) and (\ref{aa4}), we have
$$\omega_{2\alpha}=\frac{E_2(A_{12})}{A_{12}}\omega_{\alpha}:=\mu\omega_{\alpha},~~\alpha\geq 3.$$
Thus
\begin{equation*}
\begin{split}
&d\omega_{2\alpha}=d\mu\wedge\omega_{\alpha}+\mu d\omega_{\alpha}\\
&=d\mu\wedge\omega_{\alpha}-\mu C_1^2\omega_1\wedge\omega_{\alpha}+\mu^2\omega_2\wedge\omega_{\alpha}
+\mu\sum_{\gamma\geq3}\omega_{\gamma}\wedge\omega_{\gamma\alpha}.
\end{split}
\end{equation*}
Using
$d\omega_{2\alpha}-\sum_j\omega_{2j}\wedge\omega_{j\alpha}=-\frac{1}{2}\sum_{kl}R_{2\alpha
kl}\omega_k\wedge\omega_l$, we get that
$$E_(\mu)-\mu C_1=-A_{12}.$$
On the other hand, using (\ref{a1}) and (\ref{aa1}), we have
\begin{equation*}
\begin{split}
E_1(\mu)&=E_1\big[\frac{E_2(A_{12})}{A_{12}}\big]=\frac{E_1E_2(A_{12})}{A_{12}}-\frac{E_2(A_{12})E_1(A_{12})}{A_{12}^2}\\
&=\frac{E_1E_2(A_{12})}{A_{12}}-(n+1)\frac{E_2(A_{12})C_1}{A_{12}}\\
&=\frac{(E_2E_1+C_1E_2)(A_{12})}{A_{12}}-(n+1)\frac{E_2(A_{12})C_1}{A_{12}}\\
&=\frac{E_2[(n+1)A_{12}C_1]}{A_{12}}-n\frac{E_2(A_{12})C_1}{A_{12}}\\
&=(n+1)C_{1,2}+\frac{E_2(A_{12})C_1}{A_{12}},
\end{split}
\end{equation*}
which implies that
$$(n+1)C_{1,2}=-A_{12}.$$
This is a contradiction by $A_{12}=C_{1,2}$. Therefore $A_{12}=0$ and we finish the proof.
\end{proof}
By Lemma \ref{le1} and equation (\ref{cb}), we can derive that $dC=0$. Combining the results in \cite{lin} and Lemma \ref{metric} we finish the proof of Theorem \ref{the1}.

\section{Global rigidity of M\"{o}bius scalar curvature}
A hypersurface in $\mathbb{S}^{n+1}$ is called a M\"{o}bius isoparametric hypersurface if its M\"{o}bius form vanishes and
all the eigenvalues of the M\"{o}bius second fundamental
form $B$ with respect to $g$ are constants. In \cite{lih2}, authors gave the following classification theorem.
\begin{theorem}\cite{lih2}
Let $f: M^n\to \mathbb{S}^{n+1}$ be a M\"{o}bius isoparametric hypersurface with
two distinct principal curvatures. Then $f$ is M\"{o}bius equivalent to an open part of one of the
following M\"{o}bius isoparametric hypersurfaces in $\mathbb{S}^{n+1}$:\\
(i) the standard torus $\mathbb{S}^k(r)\times \mathbb{S}^{n-k}(\sqrt{1-r^2})$;\\
(ii) the image of $\sigma$ of the standard cylinder $\mathbb{S}^k(1)\times \mathbb{R}^{n-k}\subset \mathbb{R}^{n+1}$;\\
(iii) the image of $\tau$ of the standard $\mathbb{S}^k(r)\times \mathbb{H}^{n-k}(\sqrt{1+r^2})$ in $\mathbb{H}^{n+1}$.
\end{theorem}
To prove Theorem \ref{the2}, we only need to prove $C=0$. The way of the proof  is to use divergence theorem. First, we need some local computation.
\begin{lemma}\label{le2}
Let $f: M^n\rightarrow \mathbb{S}^{n+1}$ $(n\geq 4)$ be a conformally flat
hypersurface without umbilical points everywhere. If the M\"{o}bius scalar curvature $R$ is constant.
then under the local orthonormal
basis $\{E_1,\cdots,E_n\}$ in Lemma \ref{le1}, we have
\begin{equation}\label{aa11}
\begin{split}
&a_1=\frac{2n-1}{2n^2}-\frac{R}{2(n-1)(n-2)}+\frac{n-1}{n-2}(C_{1,1}-C_1^2),\\
&a_2=\frac{R}{2(n-1)(n-2)}-\frac{1}{2n^2}-\frac{1}{n-2}(C_{1,1}-C_1^2),\\
&A_{\alpha\alpha,1}=\frac{R}{(n-1)(n-2)}C_1-\frac{n}{n-2}(C_1C_{1,1}-C_1^3).
\end{split}
\end{equation}
Except these coefficients $A_{11,1}, A_{1\alpha,\alpha}$ and $A_{\alpha\alpha,1}$ the coefficients of $\nabla A$
are equal to zero.
\end{lemma}
\begin{proof}
The first and second equation in  (\ref{aa11}) can derive directly from the equation $tr(A)=a_1+(n-1)a_2=\frac{1}{2n}+\frac{R}{2(n-1)}$
and $R_{1\alpha1\alpha}=-\frac{n-1}{n^2}+a_1+a_2=C_{1,1}-C_1^2$ in (\ref{c1}).

From (\ref{aa}), we can get
\begin{equation}
A_{1\alpha,\alpha}=(a_2-a_1)C_1=[\frac{R}{(n-1)(n-2)}-\frac{1}{n}]C_1-\frac{n}{n-2}(C_1C_{1,1}-C_1^3).
\end{equation}
By (\ref{equa1}), we have $A_{\alpha\alpha,1}=A_{1\alpha,\alpha}+\frac{1}{n}C_1$. Combining above equation we get the third  equation in (\ref{aa11}).

Since $tr(A)=a_1+(n-1)a_2$ is constant, we have
$A_{11,1}=-(n-1)A_{\alpha\alpha,1}.$
Thus, by lemma \ref{le1}, we know that except these coefficients $A_{11,1}, A_{1\alpha,\alpha}$ and $A_{\alpha\alpha,1}$ the coefficients of $\nabla A$
are equal to zero.
\end{proof}
\begin{lemma}\label{lec2}
Let $f: M^n\rightarrow \mathbb{S}^{n+1}$ $(n\geq 4)$ be a conformally flat
hypersurface without umbilical points everywhere. If the M\"{o}bius scalar curvature $R$ is constant.
then under the local orthonormal
basis $\{E_1,\cdots,E_n\}$ in Lemma \ref{le1}, we have
\begin{equation}\label{cc11}
\begin{split}
&C_{1,11}=E_1(C_{1,1})=(n+2)C_1C_{1,1}-nC_1^3-\frac{R}{n-1}C_1,\\
&C_{\alpha,\alpha1}=E_1(C_{\alpha,\alpha})=-2C_1C_{1,1},~~C_{1,\alpha\alpha}=C_{\alpha,1\alpha}=-(C_{1,1}+C_1^2)C_1.
\end{split}
\end{equation}
Except these coefficients $C_{1,11}, C_{1,\alpha\alpha}$ and $C_{\alpha,\alpha1}$ the coefficients of $\nabla^2C$
are equal to zero.
\end{lemma}
\begin{proof}
Since $(A_{ij})=diag\{a_1,a_2,\cdots,a_2\}$ under the local orthonormal
basis, we have $A_{\alpha\alpha,1}=E_1(A_{\alpha\alpha})=E_1(a_2)=-\frac{1}{n-2}(C_{1,11}-2C_1C_{1,1})$ by Lemma \ref{le2}, combining the first equation in (\ref{aa11}),
we get the first equation in (\ref{cc11}).

By the equation (\ref{c2}) and the equation (\ref{a2}), $(C_{i,j})=diag\big(C_{1,1}, -C_1^2,\cdots,-C_1^2\big)$ under the local orthonormal basis, thus we have
\begin{equation}\label{cij}
C_{\alpha,\alpha1}=E_1(C_{\alpha,\alpha})=-2C_1C_{1,1},~~C_{1,\alpha\alpha}=C_{\alpha,1\alpha}=-(C_{1,1}+C_1^2)C_1.
\end{equation}
And the rest coefficients of $\nabla^2C$ are zero.
\end{proof}
Since the hypersurface is conformally flat, the Schouten tensor $S=\sum_{ij}S_{ij}\omega_i\otimes\omega_j$
is a Codazzi tensor (i.e., $S_{ij,k}=S_{ik,j}$), which defined by
$$S_{ij}=R_{ij}-\frac{R}{2(n-1)}\delta_{ij}.$$
Noting that the scaler curvature $R$ is constant,  $tr(A)$ and $tr(S)$ are constant by the equation (\ref{ba}).
Furthermore, we have
\begin{equation}\label{aa5}
\sum_jA_{ij,j}=-\sum_jB_{ij}C_j,~~~\sum_jS_{ij,j}=0.
\end{equation}

Under the local orthonormal basis  $\{E_1,\cdots,E_n\}$ in Lemma \ref{le1}, we have
$$(S_{ij})=diag(S_1,S_2,\cdots,S_2),$$
\begin{equation}\label{sic}
S_1=-\frac{(2n-1)(n-2)}{2n^2}+(n-2)a_1,~~~~S_2=\frac{n-2}{2n}+(n-2)a_2.
\end{equation}
Thus we have
\begin{equation}\label{sij}
\begin{split}
&S_{\alpha\alpha,1}=S_{1\alpha,\alpha}=(n-2)A_{\alpha\alpha,1}=\frac{R}{(n-1)}C_1-n(C_1C_{1,1}-C_1^3),\\
&S_{11,1}=-(n-1)(n-2)A_{\alpha\alpha,1}=-RC_1+n(n-1)(C_1C_{1,1}-C_1^3).
\end{split}
\end{equation}
\begin{lemma}\label{les2}
Let $f: M^n\rightarrow \mathbb{S}^{n+1}$ $(n\geq 4)$ be a conformally flat
hypersurface without umbilical points everywhere. If the M\"{o}bius scalar curvature $R$ is constant.
then under the local orthonormal
basis $\{E_1,\cdots,E_n\}$ in Lemma \ref{le1}, the coefficients of $\nabla^2S$ satisfy
\begin{equation}\label{ss11}
\begin{split}
&S_{11,11}=-RC_{1,1}+n(n-1)C_{1,1}^2+n(n-1)^2C_1^2C_{1,1}-n^2(n-1)C_1^4-nRC_1^2,\\
&S_{11,\alpha\alpha}=\frac{(n+1)R}{n-1}C_1^2-n(n+1)[C_1^2C_{1,1}-C_1^4],~~S_{\alpha\alpha,11}=-(n-1)S_{11,11},\\
&S_{\alpha\alpha,\alpha\alpha}=3S_{\alpha\alpha,\beta\beta}=3\{\frac{-R}{n-1}C_1^2+n[C_1^2C_{1,1}-C_1^4]\},~~\alpha\neq\beta.
\end{split}
\end{equation}
\end{lemma}
\begin{proof}
Since $(S_{ij})=diag(S_1,S_2,\cdots,S_2),$ we know that except these coefficients $S_{11,1},$ $ S_{1\alpha,\alpha}$ and $S_{\alpha\alpha,1}$ the coefficients of $\nabla S$
are equal to zero. Using the definition of the second covariant derivative of $S$, we can compute these equations in (\ref{ss11}).
\end{proof}
Since $E_1$ is principal vector corresponding the eigenvalue $\frac{n-1}{n}$ of the M\"{o}bius second fundamental form $B$, the $C_1=C(E_1)$, $C_{1,1}=\nabla C(E_1,E_1)$
are well-defined functions on $M^n$ up to a sign.
\begin{lemma}\label{le3}
Let $f: M^n\rightarrow \mathbb{S}^{n+1}$ $(n\geq 4)$ be a compact conformally flat
hypersurface without umbilical points everywhere. If the M\"{o}bius scalar curvature $R$ is constant.
then
\begin{equation}\label{le33}
\begin{split}
&\int_{M^n}C_1^2C_{1,1}dV_g=\frac{n-1}{3}\int_{M^n}|C|^4dV_g,\\
&\int_{M^n}C_{1,1}^2dV_g=\int_{M^n}|C|^4dV_g+\frac{R}{n-1}\int_{M^n}|C|^2dV_g.
\end{split}
\end{equation}
\end{lemma}
\begin{proof}
Using the coefficients of the tensor $C$ and $S$,
 we define two smooth vector fields
$$X_S=\sum_{ij}C_iS_{ij}E_j, ~~X_C=\sum_{ij}C_iE_i.$$
From Lemma \ref{le1} and the equation (\ref{aa11}), we can get  the divergence of $X_S, X_C$,
\begin{equation*}
\begin{split}
&div X_C=\sum_iC_{i,i}=C_{1,1}-(n-1)C_1^2,\\
&div X_S=(n-1)(C_{1,1}^2-C_1^4-\frac{R}{n-1}C_1^2)+\frac{R}{2(n-1)}div X_C.
\end{split}
\end{equation*}
Since the hypersurface is compact, we have
\begin{equation}\label{int1}
\begin{split}
&\int_{M^n}C_{1,1}dV_g=(n-1)\int_{M^n}|C|^2dV_g,\\
&\int_{M^n}C_{1,1}^2dV_g=\int_{M^n}|C|^4dV_g+\frac{R}{n-1}\int_{M^n}|C|^2dV_g,
\end{split}
\end{equation}
On the other hand, we compute $\triangle |C|^2$,
\begin{equation*}
\begin{split}
&\triangle|C|^2=\sum_i(E_iE_i-\nabla_{E_i}E_i) |C|^2=\sum_i(E_iE_i-\nabla_{E_i}E_i)C_1^2\\
&=E_1E_1(C_1^2)-\sum_i\nabla_{E_i}E_i(C_1^2)=2C_{1,1}^2+2C_1C_{1,11}-2(n-1)C_1^2C_{1,1}\\
&=2C_{1,1}^2+6C_1^2C_{1,1}-2nC_1^4-\frac{2R}{n-1}C_1^2.
\end{split}
\end{equation*}
Since the hypersurface is compact, we have
$$\int_{M^n}C_{1,1}^2dV_g+3\int_{M^n}C_1^2C_{1,1}dV_g-n\int_{M^n}C_1^4dV_g-\frac{R}{n-1}\int_{M^n}|C|^2dV_g=0.$$
Combining the equation (\ref{int1}), we can derive the first equation in (\ref{le33}).
\end{proof}
\begin{lemma}\label{inte01}
Let $f: M^n\rightarrow \mathbb{S}^{n+1}$ $(n\geq 4)$ be a compact conformally flat
hypersurface without umbilical points everywhere. If the M\"{o}bius scalar curvature $R$ is constant.
then
\begin{equation}\label{inte011}
\int_{M^n}\big\{C_{1,1}^3+(n+5)C_1^2C_{1,1}^2-(2n+5)C_1^4C_{1,1}+(n-1)C_1^6-\frac{2R}{3}C_1^4\big\}dV_g=0.
\end{equation}
\end{lemma}
\begin{proof}
 Using the coefficients of the tensor $C$ and $S$,
 we define a smooth vector field
$$Y_S=\sum_{ijk}C_{i,j}S_{ij,k}E_k.$$
Using Lemma \ref{lec2} and the equation (\ref{ss11}), we compute  the divergence of $Y_S$,
\begin{equation*}
\begin{split}
&\frac{1}{n(n-1)}div Y_S=\sum_{ijk}C_{i,jk}S_{ij,k}+\sum_{ijk}C_{i,j}S_{ij,kk},\\
&=C_{1,1}^3+(n+5)C_1^2C_{1,1}^2-(2n+5)C_1^4C_{1,1}+(n-1)C_1^6+\frac{(2n-1)R}{n(n-1)}C_1^4\\
&-\frac{R}{n(n-1)}C_{1,1}^2-\frac{2(n+3)R}{n(n-1)}C_1^2C_{1,1}+\frac{R^2}{n(n-1)^2}C_1^2.
\end{split}
\end{equation*}
Integrating this equation and using (\ref{le33}), we can derive the second equation in (\ref{inte011}).
\end{proof}

\begin{lemma}\label{inte1}
Let $f: M^n\rightarrow \mathbb{S}^{n+1}$ $(n\geq 4)$ be a compact conformally flat
hypersurface without umbilical points everywhere. If the M\"{o}bius scalar curvature $R$ is constant.
then
\begin{equation}\label{inte11}
\begin{split}
&\int_{M^n}\big\{C_1^2C_{1,1}^2+C_1^4C_{1,1}-\frac{n}{3}C_1^6-\frac{R}{3(n-1)}C_1^4\big\}dV_g=0,\\
&\int_{M^n}\big\{C_{1,1}^3+(n-1)C_1^2C_{1,1}^2-(2n+1)C_1^4C_{1,1}+(n+1)C_1^6+\frac{2(n+1)R}{3(n-2)}C_1^4\big\}dV_g=0.
\end{split}
\end{equation}
\end{lemma}
\begin{proof} Using (\ref{cc11}),
\begin{equation*}
\begin{split}
\frac{1}{4}\triangle |C|^4&=\frac{1}{4}\sum_i(E_iE_i-\nabla E_i{E_i})|C|^4=3C_1^2C_{1,1}^2+C_1^3C_{1,11}-(n-1)C_1^4C_{1,1}\\
&=3C_1^2C_{1,1}^2+3C_1^4C_{1,1}-nC_1^6-\frac{R}{n-1}C_1^4.
\end{split}
\end{equation*}
Since the hypersurface is compact, then
$$\int_{M^n}\{3C_1^2C_{1,1}^2+3C_1^4C_{1,1}-nC_1^6-\frac{R}{n-1}C_1^4\}dV_g=0.$$
Since $S_{ij}$ is a Codazzi tensor and $tr(S)$ is constant, we can compute $\triangle |S|^2$ by (\ref{sij}),
\begin{equation*}
\begin{split}
&\frac{1}{n^2(n-1)}\triangle |S|^2=\frac{1}{2n^2(n-1)}\{\sum_{ijk}|S_{ij,k}|^2+\frac{1}{2}\sum_{ij}(S_i-S_j)^2R_{ijij}\}\\
&=C_{1,1}^3-(2n+1)C_1^4C_{1,1}+(n-1)C_1^2C_{1,1}^2-\frac{2R}{n-1}C_1^2C_{1,1}-\frac{2R}{n(n-1)}C_{1,1}^2\\
&+(n+1)C_1^6+\frac{2(n+1)R}{n(n-1)}C_1^4+\frac{(n+1)R^2}{n^2(n-1)^2}C_1^2+\frac{R^2}{n^2(n-1)^2}C_{1,1}.
\end{split}
\end{equation*}
Integrating this equation and using (\ref{le33}), we can derive the second equation in (\ref{inte11}).
\end{proof}
\begin{lemma}\label{inte2}
Let $f: M^n\rightarrow \mathbb{S}^{n+1}$ $(n\geq 4)$ be a compact conformally flat
hypersurface without umbilical points everywhere. If the M\"{o}bius scalar curvature $R$ is constant.
then
\begin{equation}\label{inte21}
\begin{split}
&\int_{M^n}\big\{C_{1,1}^3+\frac{5n+4}{2}C_1^2C_{1,1}^2-3(n+1)C_1^4C_{1,1}+\frac{n}{2}C_1^6-\frac{(7n-10)R}{3(n-2)}C_1^4\big\}dV_g=0£¬\\
&\int_{M^n}\big\{C_{1,1}^3+\frac{5n+16}{2}C_1^2C_{1,1}^2-3(n-1)C_1^4C_{1,1}-\frac{3n}{2}C_1^6-\frac{(7n+2)R}{6(n-1)}C_1^4\big\}dV_g=0.
\end{split}
\end{equation}
\end{lemma}
\begin{proof}
Using the coefficients of the tensor $C$ and $A$, we have two following smooth functions,
$$|C|^2_A=\sum_{ij}C_iA_{ij}C_j=C_1^2a_1, ~~|C|^2_C=\sum_{ij}C_iC_{i,j}C_j=C_1^2C_{1,1}.$$
Next we compute $\triangle (|C|^2_A)$ and $\triangle(|C|^2_C)$.
\begin{equation*}
\begin{split}
&\frac{n-2}{2(n-1)}\triangle (|C|^2_A)=\frac{n-2}{2(n-1)}(E_1E_1(C_1^2a_1)-(n-1)C_1E_1(C_1^2a_1))\\
&=C_{1,1}^3-3(n+1)C_1^4C_{1,1}+\frac{(5n+4)}{2}C_1^2C_{1,1}^2+\frac{n}{2}C_1^6\\
&+\frac{n-2}{2(n-1)}[\frac{2n-1}{n^2}-\frac{R}{(n-1)(n-2)}]C_{1,1}^2-\frac{n-2}{2(n-1)}[\frac{2n-1}{n}-\frac{(2n-1)R}{(n-1)(n-2)}]C_1^4\\
&\frac{n-2}{2(n-1)}[\frac{3(2n-1)}{n^2}-\frac{(7n-4)R}{(n-1)(n-2)}]C_1^2C_{1,1}\\
&-\frac{n-2}{2(n-1)}[\frac{2n-1}{n^2}-\frac{R}{(n-1)(n-2)}]\frac{R}{n-1}C_1^2.
\end{split}
\end{equation*}
Integrating this equation and using (\ref{le33}), we can derive the first equation in (\ref{inte21}).
\begin{equation*}
\begin{split}
&\frac{1}{2}\triangle (|C|^2_C)=\frac{1}{2}(E_1E_1(C_1^2C_{1,1})-(n-1)C_1E_1(C_1^2C_{1,1}))\\
&=C_{1,1}^3-3(n-1)C_1^4C_{1,1}+\frac{(5n+16)}{2}C_1^2C_{1,1}^2-\frac{3n}{2}C_1^6\\
&-\frac{7R}{2(n-1)}C_1^2C_{1,1}-\frac{3R}{2(n-1)}C_1^4.
\end{split}
\end{equation*}
Integrating this equation and using (\ref{le33}), we can derive the second equation in (\ref{inte21}).
\end{proof}
Now we combine these equation system in (\ref{inte011}), (\ref{inte11}) and (\ref{inte21}), we can derive that
$\int_{M^n}C_1^4dV_g=0,$ which implies that $C_1=0$ and the M\"{o}bius form vanishes. Thus we finish the proof of Theorem \ref{the2}.

\end{document}